\documentclass{amsart}
\usepackage[T1]{fontenc}
\usepackage[utf8]{inputenc}
\usepackage[english]{babel}
\usepackage{amsthm}
\usepackage{amssymb}
\usepackage{amsmath}
\usepackage[all]{xy}
\usepackage{mathrsfs}
\usepackage{verbatim}
\usepackage{enumitem}
\usepackage{stmaryrd}
\usepackage{pgf,tikz}
\usetikzlibrary{arrows}
\newtheorem{teo}{Theorem}[section]
\newtheorem{thmx}{Theorem}

\newtheorem{lem}[teo]{Lemma}     %{Lemma}[section]
\newtheorem{prop}[teo]{Proposition}
\newtheorem{cor}[teo]{Corollary}

\theoremstyle{definition}

\theoremstyle{claim}

\newcommand{\normal}{\trianglelefteq}

\newcommand{\abs}[1]{\lvert#1\rvert}

\begin{document}
\title{Finite non-solvable groups whose real degrees are prime-powers}
\author{lorenzo bonazzi}
\address{Department of Mathematics and Informatics, University of Florence}
\email{lorenzo.bonazzi@unifi.it}
	%\maketitle
	
		\begin{abstract}
		We present a description of non-solvable groups in which all real irreducible character degrees are prime-power numbers.
	\end{abstract}
\maketitle
	\section{Introduction}
	Let $G$ be a finite group. It is well known that $cd(G)$, the set of the degrees of all irreducible characters, has great impact on the structure of $G$. Manz in \cite{manz1985non-solv} and \cite{manz1985solv} described the solvable and non-solvable groups in which all the real irreducible characters have prime-power degrees. In this paper we study the same problem for real characters in the non-solvable case. We give a structural description of non-solvable groups $G$ such that $cd_{rv}(G)$, the set of the degrees of all real irreducible characters, consists of prime-power numbers. In the following, $Rad(G)$ is radical subgroup and $G^{(\infty)}$ is the last term of the derived series.
	\begin{thmx}\label{theoremA} Let $G$ be a finite non-solvable group and suppose that $cd_{rv}(G)$ consists of prime-power numbers. Then $Rad(G)=H\times O$ for a group $O$ of odd order and a $2$-group $H$ of Chillag-Mann type. Furthermore, if $K=G^{(\infty)}$, then one of the following holds.
		\begin{itemize}
			\item[i)] $G =K\times Rad(G)$ and $K$ is isomorphic to $A_5$ or $L_2(8)$;
			\item[ii)] $G=(KH)\times O$ where $K\simeq SL_2(5)$, $HK=H\Ydown K$ and $K\cap H=Z(K)<H$. 
		\end{itemize}
	\end{thmx}
About the point $ii)$, we remark that if $G$ is the the \texttt{SmallGroup(240,93)}, then $K \simeq SL_2(5)$, $\abs{H}=4$ and $K\cap H=Z(K)$.%The main obstacle to get a full description in the case $ii)$ is the extension of $2^9\cdot SL_2(5)$, that can't be reached by the \texttt{GAP}'s library \texttt{PerfectGroups}.

As a Corollary, we get control on the set of real character degrees. We recall that $cd_{rv,2'}(G)$ is the set of odd real character degrees of a finite group $G$.
\begin{thmx}\label{theoremB}
	Let $G$ a non-solvable group such that $cd_{rv}(G)$ consists of prime-power numbers. Then either
	\begin{itemize} \item[i)] $cd_{rv}(G)=cd_{rv}(L_2(8))$ or \item[ii)]$cd_{rv,2'}(G)=cd_{rv,2'}(A_5)$.
		\end{itemize}
\end{thmx}

	\section{Preliminar results and Lemmas}
	Chillag and Mann are among the first authors that studied $cd_{rv}(G)$. They characterized the groups $G$ such that $cd_{rv}(G)=\{1\}$, namely where all real irreducible characters are linear. Now these groups are commonly known as groups of \emph{Chillag-Mann type}.
	\begin{teo} \cite[Theorem 1.1]{chillag-mann1998} \label{chillag-mann-1998-thm1.1}
		Let $G$ a finite group of Chillag-Mann type. Then $G=O\times T$ where $O$ is a group of odd order and $T$ is a $2$-group of Chillag-Mann type.
	\end{teo}
	%It is now clear that the study of groups of Chillag-Mann type boils down to the study of $2$-groups of Chillag-Mann type, for which we report some properties.
%	
	%\begin{lem}\label{chillag-mann%-linear}
		%Let $T$ be a $2$-group and $\lambda$ a linear real irreducible character of $T$. Then $\Phi(T)\le \ker(\lambda)$.
%	\end{lem}
%	If $\Gamma$ is a graph, we denote $n(\Gamma)$ the number of connected components of $\Gamma$.
%	\begin{teo}\cite[Theorem 5.1]{dnt2008} \label{dnt2008thm5.1}
%		Let $G$ be a finite non-solvable group. Then
%$n(\Delta_{rv}(G))\le 3$.
%	\end{teo}	
	One other important contribution, was given by Dolfi, Navarro and Tiep in \cite{dnt2008}. In their paper, appears version for real characters of the celebrated Ito-Michler Theorem for the prime $p=2$. Recall that $Irr_{rv}(G)$ denotes the set of irreducible real valued character of $G$.
	\begin{teo} \cite[Theorem A]{dnt2008} \label{dnt2008thmA}
		Let $G$ be a finite group and $T \in Syl_2(G)$. Then $2 \nmid \chi(1) $ for every non-linear $\chi \in Irr_{rv}(G)$ if and only if $T \normal G$ and $T$ is of Chillag-Mann type.
	\end{teo}
	The corresponding condition for an odd prime $p$ was studied by Tiep in \cite{tiep2015} and Isaacs and Navarro in \cite{isaacs-navarro-2012}. Though a partial result, the techniques involved are deep. This confirm the special role of the prime $2$ in the study od real character degrees.
	\begin{teo}	\cite[Theorem A]{tiep2015} \label{tiep2015thmA}
		Let $G$ be a finite group and $p$ be a prime. Suppose that $p \nmid \chi(1)$ for every $\chi \in Irr_{rv}(G)$ with Schur-Frobenius indicator $1$. Then $O^{p'}(G)$ is solvable; in particular, $G$ is $p$-solvable.
	\end{teo}
	Navarro, Sanus and Tiep gave a version for real characters of Thompson's Theorem for the prime $2$ in \cite{nst2009}. Their work includes also a characterization of groups in which the real character degrees are powers of $2$.
	\begin{teo} \cite[Theorem A]{nst2009} \label{nst2009thmA}
		Let $G$ a finite group and suppose that $2$ divides $\chi(1)$ for all every real non-linear irreducible character of $G$. Then $G$ has a normal $2$-complement.
	\end{teo}
	
	The next two Lemmas appears on \cite{dps2008}.
	\begin{lem}\label{dps1.4}
		Let be $N$ a normal subgroup of $G$ and $\chi \in Irr_{rv}(G)$. The following hold.
		\begin{itemize}
			\item[i)] if $\chi(1)$ is odd, then $N\le \ker (\chi)$;
			\item[ii)] if $\abs{N}$ odd and $N$ centralizes a Sylow $2$-subgroup of $G$, then $N \le \ker(\chi)$.
		\end{itemize}
	\end{lem}
	\begin{proof}
	Point $ii)$ is \cite[1.4]{dps2008}. Point $i)$ follows from the discussion before \cite[1.4]{dps2008}, keeping in mind that a group of odd order does not have any real non-trivial character.
	\end{proof}
%
%	\begin{lem}  \label{dps1.4}
%		Let $G$ be a group and $N \normal G$ with 
%	\end{lem}
	Let be $N$ is a normal subgroup of $G$ and $\theta \in Irr_{rv}(N)$. The next Lemmas provide some sufficient conditions for the existence of a real character of $G$ above $\theta$.
	\begin{lem}\cite[2.1 and 2.2]{nt2007}\label{realextension1}
		Let $N$ a normal subgroup of a group $G$ and $\theta \in Irr_{rv}(N)$. If $[G:N]$ is odd, then $\theta$ allows a unique real-valued extension to $I_G(\theta)$. Furthermore, there exists a unique real-valued character $\chi \in Irr_{rv}(G\mid\theta)$.
	\end{lem}	
	\begin{lem}\cite[2.3]{nt2007}\label{realextension} Let $G$ a finite group and $N \normal G$. Suppose that there is $\theta \in Irr_{rv}(G)$ such that $ \theta(1)$ is odd and $o(\theta)=1$. Then $\theta$ extends to a character $\varphi \in Irr_{rv}(I_G(\theta))$ and $\chi=\varphi^G \in Irr_{rv}(G\mid\theta)$.
	\end{lem}
	
	\begin{lem} \label{tensor-induction}
		Let $N$ a minimal normal subgroup of a group $G$, $N=S_1 \times \dots \times S_n$ where $S \simeq S$ is a non-abelian simple group. Let $\sigma \in Irr_{rv}(S)$ and suppose that $\sigma $ extends to a real character of $Aut(S)$. Then $\sigma \times \dots \times \sigma$ extends to a real character of $G$.
	\end{lem}
	\begin{proof}The extension $\chi$ is constructed in \cite[Lemma 5]{bianchi-chillag-lewis-pacifici2007}. We see that if $\sigma$ takes real values, then also $\chi$ does.
	\end{proof}
	The technique used in the proof of Lemma \ref{tensor-induction} is  known as \emph{tensor induction}, for further details see \cite[Section 4]{isaacs1982-corrispondence}.

		\begin{lem} \cite[1.6]{dps2008} \label{dps1.6}
		Let $G$ a finite group that acts by automorphism on the group $M$. For every involution $xC_G(M) \in G/C_G(M)$ there exists a non trivial character $\mu \in Irr(M)$ such that $\mu^x =\bar \mu$.
	\end{lem}
	\section{Proofs}
	In the following, we denote an integer a \emph{composite number} if it is divisible by more than one prime. If $p$ is a prime, we denote by $p^*$ a general positive integer that is a power of $p$.
	Moreover, $Rad(G)$ is the solvable radical of $G$, namely the largest solvable normal subgroup of $G$.  
	\begin{teo} \label{noedgesnorad}	
		Let $G$ be a finite non-solvable group such that $cd_{rv}(G)$ consists of prime-power numbers. If $Rad(G)=1$, then $G$ is isomorphic to $A_5$ or $PSL_2(8)$.
	\end{teo}
	\begin{proof}
	Let be $M$ a minimal normal subgroup of $G$. Then $M=S_1 \times \dots \times S_n$ is the product of simple groups, all isomorphic to to a simple group $S$. Since $Rad(G)=1$, the group $S$ is non-abelian.\newline
	
	\textbf{Step 1}: $S$ is isomorphic to one of the following groups
	\[ A_5, A_6, PSL_2(8), PSL_3(3), PSp_4(3), PSL_2(7), PSU_3(3),PSL_2(17).\]
	Let $p \in \pi(M)$. Since $M$ is minimal normal in $G$, we have $M \le O^{p'}(G)$, so $O^{p'}(G)$ is non-solvable. By Theorem \ref{tiep2015thmA} there is a real irreducible character $\chi$ of $G$ such that $p \mid \chi(1)$. By the hypothesis, $\chi(1)=p^*>1$. This means that for every prime $p \in \pi(M)$, there is $\chi \in Irr_{rv}(G)$ such that $\chi(1)=p^*>1$. By Theorem A of \cite{dnt2008}, if $\Delta_{rv}(G)$ is the prime graph on real character degrees of $G$, then the number of connected components of $\Delta_{rv}(G)$ is at most three. In our hypotheses, $\Delta_{rv}(G)$ consists in isolated vertices and hence the number of primes that appear as divisors of the degree of some real irreducible character, is at most $3$. It follows that $M$, and hence $S$, is divisible by exactly $3$ primes. Now, by Lemma $2.1$ in \cite{zhang-chen-chen-liu-2009}, the simple groups having order divided by exactly $3$ distinct primes are those stated. \newline
	
	\textbf{Step 2}: $S$ is isomorphic to one of the following groups: $A_5, PSL_2(8), A_6$\newline 
	
	If $S \in \{ PSp_4(3), PSL_3(3), PSU_3(3)\}$ then there is a non-linear character $\sigma \in Irr_{rv}(S)$ such that $\sigma(1)$ is an odd composite number. Let $\theta=\sigma\times \dots \times \sigma \in Irr_{rv}(M)$. Then $ 2 \nmid \theta(1)$ and $o(\theta)=1$, since $M$ is perfect. So, by Lemma \ref{realextension}, there is $\chi \in Irr_{rv}(G\mid\theta)$. As $\theta(1)$ divides $\chi(1)$, the degree of $\chi$ is a composite number, against the hypothesis.
	
	Suppose that $S \in \{PSL_2(7), PSL_2(17)\}$. There is a real character $\sigma \in Irr_{rv}(S)$ such that $\sigma(1)$ is a composite number and $\sigma$ extends to a real character of $A=Aut(S)$. By tensor induction (Lemma \ref{tensor-induction}), the character $\theta=\sigma \times \dots \times \sigma$ extends to a real character $\chi \in Irr_{rv}(G)$. Again $\chi(1)=\theta(1)=\sigma(1)^n$ is a composite number.\newline
	
	\textbf{Step 3}: $n=1$ and $M$ is a simple group.\newline
	
	The only left possibilities are $S \in \{A_5, PSL_2(8), A_6\}$. Checking the character table of these groups, we see that there are two non-linear characters $ \sigma, \rho \in Irr_{rv}(S)$ such that $\sigma(1)=p^*>1$ and $\rho(1)=q^*>1$ for $p,q$ odd distinct primes. Let $\theta=\sigma \times 1 \times \dots \times 1 \in Irr_{rv}(M)$. Since $o(\theta)=1$ and $\theta(1)$ is odd, the character $\theta$ extends to a character $\varphi \in Irr_{rv}(I_G(\theta))$ by Lemma \ref{realextension} and $\chi=\varphi^G$ has degree $p^*$, hence $[G:I_G(\theta)]=p^*>1$. Since $I_G(\theta) \le N_G(S_1)$, we have that
	\[n=[G:N_G(S_1)] \text{ divides } [G:I_G(\theta)]=p^*>1,
	\]
	so $n=p^*$. By the same argument with $\rho$ in place of $\sigma$, we get that $n=q^*$ and $n \mid (p^*,q^*)=1$. \newline
	
	\textbf{Step 4}: $C_G(M)=1$.\newline
	
	Suppose, by contradiction, that $C_G(M)>1$ and take $N$ a minimal normal subgroup of $G$ contained in $C_G(M)$. For the same arguments used on $M$, we have that $N$ is simple and is isomorphic to one of the following groups $A_5, PSL_2(8), A_6$. As before, take $\sigma \in Irr_{rv}(M)$ with $\sigma(1)=p^*$ and $ \rho \in Irr_{rv}(N)$ with $\rho(1)=q^*$ for $p,q$ odd distinct primes. Note that $[M,N] \le M \cap N \le M \cap C_G(M)=1$ since $M$ is simple and non abelian. So $MN=M \times N$ is perfect normal in $G$ and $\theta=\sigma \times \in Irr_{rv}(MN)$. Note that $o(\theta)=1$ and $2 \nmid \theta(1)$. By Lemma \ref{realextension} there is $\chi \in Irr_{rv}(G\mid \theta)$, and this is impossible, since $\chi(1)$ is not a composite number.\newline
	
	\textbf{Conclusion}: we proved, so far, that: $S \le G \le Aut(S)$ and that \[ S \in \{A_5 , A_6, PSL_2(8)\}.\]
	Now, $S$ cannot be the alternating group $ A_6$ because each of the $5$ subgroups between $S$ and $Aut(S)$ has a rational irreducible character of degree $10$ (it is possible check this with the software \texttt{GAP}), so $S \in \{A_5, PSL_2(8)\}$. In any of these cases, $[Aut(S):S]$ is a prime number and there is only one subgroup strictly above $S$, namely $Aut(S)$ itself. But both $Aut(A_5)$ and $Aut(PSL_2(8))$ have a real irreducible character with composite degree. Hence $G =A_5$ or $G=PSL_2(8)$.
	\end{proof}
	%In the previous Theorem, specific information about simple groups, their character degrees and automorphism groups is obtained using the software GAP.
	%For a finite group $G$ denote $M(G)=H^2(G,\mathbb{C})$ its Schur multiplier. Note that $\abs{M(A_5)}=2$ and $M(PSL_2(8))=1$.

	\begin{prop} \label{noedgessimple}
		Let $G$ be a finite non-solvable group such that $cd_{rv}(G)$ consists of prime-power numberss. Then $G=KR$ with $R=Rad(G)$ and $K=G^{(\infty)}$. Moreover $K\cap R=L$ is a $2$-group and $K/L$ is isomorphic to $A_5$ or $PSL_2(8)$.
	\end{prop}
	
	\begin{proof} 
	Let $K=G^{{(\infty)}}$ be the last term of the derived series of $G$ and call $\bar G=G/K\cap R$. Observe quotients preserve the hypotheses. Hence, by Theorem \ref{noedgesnorad}, $G/R$ is a simple group. Since $1<KR/R \normal G/R$, we have that $G=KR$ and $\bar K \simeq G/R$ is isomorphic to $A_5$ or $PSL_2(8)$. Moreover, $\bar G=\bar K\times\bar R$ because $[K,R]\le L$. \newline
	Suppose by contradiction that there is $\theta \in Irr_{rv}(\bar R)$ of non-trivial degree. By Theorems \ref{nst2009thmA} and \ref{dnt2008thmA}, there are two non linear characters $\phi_1, \phi_2 \in Irr_{rv}(\bar K)$ such that $\phi_1(1)$ is even and $\phi_2(1)$ is odd. If $\theta(1)$ is odd, consider $\chi=\theta\phi_1$ and if $\theta(1)$ is even, consider $\chi=\theta\phi_2$. In any case, $\chi$ is a composite number, but this is impossible. It follows that every real character of $R/L$ is linear and by Theorem \ref{chillag-mann-1998-thm1.1} $\bar R= \bar O \times \bar H$, where $O \in Hall_{2'}(R)$ and $H \in Syl_2(R)$. Write $G_0$ for the preimage in $G$ of $\bar K \bar H$, note that $G_0$ is a normal subgroup of odd index in $G$. Note that $G_0=LKH=KH$. By Lemma \ref{realextension1}, $cd_{rv}(G_0)$ consists of prime-power numbers. Moreover $K=G_0^{(\infty)}$ and $Rad(G_0)\cap K=L$. Hence we can assume that $G=G_0$. This implies that $O \le L$.\newline

	Suppose, working by contradiction, that $O>1$, namely $L$ is not a $2$-group.
	Consider $M/M_0$ the first term (from above) of a principal series of $G$ such that $M,M_0 \le L$ and $M/M_0$ is not a $2$-group. Hence $M/M_0$ is an elementary abelian $p$-group for $p$ odd and $L/M$ is a $2$-group. Possibly replacing $G$ with $G/M_0$, we can assume that $M_0=1$ and $M$ is a minimal normal subgroup of $G$. \newline
	Since $K/L$ is simple, $C_K(M)=K$ or $C_K(M) \le L$. If $C_K(M)=K$, then $M$ has a direct complement $N$ in $L$ and consider $\bar K=K/N$. Note that  $1<\bar M \le Z(\bar K)\cap \bar K'$, since $K=K'$ is perfect and hence $\abs{M}$ divides $\abs{M(G)}$ by \cite[11.20]{ctfg}, where $M(G)$ denotes the Schur multiplier of $G$. But this is impossible, since $\abs{M(A_5)}=2$ and $M(PSL_2(8))=1$.\newline 
	
	Hence $C_K(M) \le L$ and the action of $K$ on $M$ is non-trivial. Moreover $K/L$ has even order, so by Lemma \ref{dps1.6} there is an element $\lambda \in \hat{M}$ and $ x \in K$ such that $\lambda^x=\bar \lambda$. Let $I=I_G(\lambda)$ and note that $x \in N_G(I)\setminus I$, so $2$ divides $[G:I]$.\newline 
	
	Let $\bar I=I/Ker(\lambda)$ (we remark that "bar" notation here is not the same as in first part of the proof) and observe that $\bar M \le Z(\bar I)$. Take $P  \in Syl_p(I)$; since the index of $K$ in $G$ is a $2$-power, every subgroup of $G$ with odd order is contained in $K$, hence $P \le K$. Moreover, $\bar M \le Z(\bar P)$, $\bar P \in Syl_p(\bar I)$ and $PL/L $ is a $p$-subgroup of the simple group $K/L$, that is isomorphic to $A_5$ or $PSL_2(8)$. Now, if $p$ is an odd prime, every Sylow $p$-subgroup of $A_5$ or $PSL_2(8)$ is cyclic (see tables \ref{tab:subgroupsA5} and \ref{tab:subgroupsPSL(2,8)}). Hence, $P/M\simeq\bar P/\bar M \simeq PL/L$ is cyclic and $\bar P$ is abelian.\newline

	Since $\bar M \le Z(\bar I)$, we have that $\bar M \nleq \bar I'$ by Theorem \cite[5.3]{fgt}. In addition $\bar M \cap \bar I'=1$ because $\bar M$ has order $p$. Write $\bar I/\bar I'=Q \times B$, where $B \in Hall_{p'}(\bar I/\bar I')$ and $Q \in Syl_p(\bar I /\bar I')$. Note that $Q$ and $B$ are $x$-invariant, as $x$ normalizes $I$. By abuse of notation, we write $M \le Q$ in the place of $\bar M \bar I'/\bar I' \le Q$. In this notation $M$ is a group of order $p$ and $\lambda$ is a faithful character of $M$. The $2$-group $\langle x \rangle$ acts on the abelian group $Q$, hence by Maschke's Theorem \cite[8.4.6]{ks} there is an $\langle x \rangle$-invariant complement $T$ for the $\langle x \rangle$-invariant subgroup $M$, so $Q=M\times T$. Let $\hat \lambda=\lambda \times 1_T \in Irr(Q)$ and $\delta=\hat \lambda \times 1_B \in Irr(\bar I /\bar I')$, we have that
	\[
	\delta^x=\hat \lambda^x \times 1_{B^x}=(\lambda^x \times 1_{T^x})\times 1_B=(\bar \lambda \times 1_T) \times 1_B= \bar \delta.
	\]
	We return to the previous notation, so $\delta$ lifts to a character of $I$, that we call again $\delta$. Note that $I<G$ as $2$ divides $[G:I]$.\newline
	
	If $IH<G$, then $IH/H$ is a proper subgroup of $G/H$ that is a simple group isomorphic to $A_5$ or $PSL_2(8)$. The maximal subgroups of these two groups are known as well as their indexes, see tables \ref{tab:subgroupsA5} and \ref{tab:subgroupsPSL(2,8)}. In particular, there always is an odd prime $q$ such that $q$ divides $[G:IH]$ and hence $2q $ divides $[G:I]$.
	Note that $\delta \in Irr(I\mid\lambda)$, so $\chi = \delta^G \in Irr(G)$. Moreover
	\[\bar \chi = (\bar \delta)^G=(\delta^x)^G=\delta^G=\chi.
	\]
	Hence $\chi$ is a real character of $G$ and $2q \mid \chi(1)$ since $2q \mid [G:I]$, and this is impossible. \newline
	
	Suppose now $IH=G$. In this case, $I/I\cap H \simeq G/H$ that is isomorphic to $A_5$ or $PSL_2(8)$. These groups have a unique rational character $\phi$ of odd degree. The element $x$ stabilizes the section $I/I\cap H$, hence by uniqueness $\phi^x=\phi$. By Gallagher Theorem \cite[6.17]{ctfg}, $\phi \delta \in Irr(I\mid \lambda)$ and by Clifford corrispondance, $\chi=(\phi\delta)^G \in Irr(G)$. Since $\phi$ is a real $x$-invariant character and $\delta^x=\bar \delta$, we have that $(\phi\delta)^x=\overline{\phi\delta}$. Hence, as before $\chi$ is a real irreducible character. Now $\theta(1) \mid \chi(1)$ and there is an odd prime $q$ such that $q$ divides $\chi(1)$. Moreover $2 \mid \chi(1)$ since $2$ divides $[G:I]$. So $\chi(1)$ is a composite number and this is impossible.\end{proof}
We give the list of maximal subgroups of $A_5$ and $PSL_2(8)$ and their indices. 
\begin{table}[h!]
	\begin{center}
		\caption{Maximal subgroups of $A_5$.}
		\label{tab:subgroupsA5}
		\begin{tabular}{l|c|r} % <-- Alignments: 1st column left, 2nd middle and 3rd right, with vertical lines in between
			$A_4$ & $D_{10}$ & $S_3$\\
			
			\hline
			$12$ & $10$ & $6$\\
			\hline
			$5$ & $6$ & $10$\\
		\end{tabular}
	\end{center}
\end{table}
\begin{table}[h!]
	\begin{center}
		\caption{Maximal subgroups of $PSL_2(8)$.}
		\label{tab:subgroupsPSL(2,8)}
		\begin{tabular}{l|c|r} % <-- Alignments: 1st column left, 2nd middle and 3rd right, with vertical lines in between
			$F_{56}$ & $D_{18}$ & $D_{14}$\\
			
			\hline
			$56$ & $18$ & $14$\\
			\hline
			$9$ & $28$ & $72$\\
		\end{tabular}
	\end{center}
\end{table}

\begin{lem}\label{extensions-simple-groups}
	Let be $K$ a perfect group and $M$ a minimal normal subgroup of $K$ that is an elementary abelian $2$-group. Suppose that $M$ is non-central in $K$ and $K/M$ is isomorphic to $L_2(8)$ or $A_5$. Then $K$ has an irreducible non-linear real character with odd composite degree .
\end{lem}
\begin{proof}
	Since $G/M$ is simple we have that $C_G(M)=M$. Suppose that $K/M$ is isomorphic to $A_5$. There are two non isomorphic irreducible $A_5$-modules $W_1,W_2$ of $A_5$ over $GF(2)$. Both have dimension $4$ and $H^2(A_5,W_1)=H^2(A_5,W_2)=0$. Hence $M$ has a complement $S$ in $K$. It is easy to construct these groups and we see that $K=M\rtimes S=W_i\rtimes A_5$ has a real irreducible character of degree $15$. Suppose now that $K/M \simeq L_2(8)$. Let be $W_1,W_2,W_3$ the non-isomorphic irreducible $L_2(8)$-modules over  $GF(2)$, where $\dim(W_1)=6,\dim(W_2)=8$ and $\dim(W_3)=12$. If $M\simeq M_i$ with $i=2,3$, then $H^2(L_2(8),W_i)=0$ and hence $M_i$ has a complement $S$ in $K$. Then, as before, we conclude observing that $W_i\rtimes L_2(8)$ has a real irreducible character of degree $63$. Suppose that $M \simeq W_1$. Then $\dim{H^2(L_2(8),W_1)}=3$. Nevertheless, there are just two perfect groups of order $2^6 \cdot \abs{L_2(8)}$. Both these groups have an irreducible real character of degree $63$.
\end{proof}
In the previous Lemma, dimensions of chomology groups and all the perfect groups of a given order is information that is accesible with the \texttt{GAP}'s functions \texttt{cohomolo} and \texttt{PerfectGroup}.
\begin{prop} \label{perfect-characterized}
Let $G$ be a finite non-solvable group and suppose that $cd_{rv}(G)$ consists of prime-power numbers. Let be $K=G^{(\infty)}$ and $R=Rad(G)$. Then $\abs{K\cap R}\le 2$ and if equality holds, then $K \simeq SL_2(5)$.
\end{prop}
\begin{proof} By Proposition \ref{noedgessimple}, we have that $N=K\cap R$ is a $2$-group. We prove that if $N>1$ then $\abs{N}=2$ and $K$ is isomorphic to $SL_2(5)$. Let be $V=N/\Phi(N)$, then $V$ a normal elementary abelian $2$-subgroup of $G/\Phi(N)$. Let $V>V_1>\dots >V_n$ a $K$-principal series of $V$. Let be $N>N_1>\dots >N_n$ such that $N_i$ the preimage in $N$ of $V_i$. Then $N/N_1$ is an irreducible $K/N$-module and $K/N$ is isomorphic $A_5$ or $L_2(8)$ by Proposition \ref{noedgessimple}. By Lemmas \ref{extensions-simple-groups} and \ref{realextension}, $N/N_1$ is central in $K/N_1$. Since $K$ is perfect, we have that $N/N_1$ is isomorphic to a subgroup of the Schur multiplier $M(K/N)$. The only possibility is $\abs{N/N_1}=2$ and $K/N_1\simeq SL_2(5)$, the Schur covering of $A_5$. Suppose by contradiction that $N_1/N_2>1$, write $\bar K=K/N_2$. Since $M(SL_2(5))=1$, $\bar N_1$ cannot be central in $\bar K$. Let $t \in K$ a $2$-element such that $\langle tN_1\rangle=Z(K/N_1)$, namely the unique central involution in $SL_2(5)$ and $\langle tN_1\rangle=O_2(K/N_1)$. Since $N_1$ is an irreducible module over $GF(2)$, we have that $t$ acts trivially on $\bar N_1$. Suppose that $\bar t^2\neq 1$, then $\langle \bar t^2\rangle$ would be a proper, non-trivial submodule of $\bar N_1$, against irreducibility. This means that $\bar t^2=1$ and hence $\langle\bar t\rangle$, that centralizes $\bar N_2$, is a minimal normal subgroup of $\bar N_2$. Observe that $\bar K/\langle t \rangle$ is a quotient of $K$ that satisfies the hypotheses of Lemma \ref{extensions-simple-groups}. Hence by Lemma \ref{realextension} we derive a contradiction.% Since $SL_2(5)$ does not satisfy the hypotheses, we have that $K\cap H<H$.
\end{proof}
We now prove Theorem \ref{theoremA}, that we restate for convenience of the reader. 
\begin{teo}\label{noedgesnonsolvmain}
Let $G$ be a finite non-solvable group and suppose that $cd_{rv}(G)$ consists of prime-power numbers. Then $Rad(G)=H\times O$ for a group $O$ of odd order and a $2$-group $H$ of Chillag-Mann type. Furthermore, if $K=G^{(\infty)}$, then one of the following holds.
\begin{itemize}
\item[i)] $G =K\times R$ and $K$ is isomorphic to $A_5$ or $L_2(8)$;
\item[ii)] $G=(KH)\times O$ where $K\simeq SL_2(5)$, $HK=H\Ydown K$ and $K\cap H=Z(K)<H$. 
\end{itemize}
\end{teo}
\begin{proof} 
By Proposition \ref{perfect-characterized} and Proposition \ref{noedgessimple}, if $K=G^{(\infty)}$ and $R=Rad(G)$, then $G=KR$, and either $K\cap R=1$ and $K$ is simple isomorphic to $A_5$ or $L_2(8)$ or $K\simeq SL_2(5)$ and $K\cap R=Z(K)$. In the first case, $i)$ follows. Suppose $K=SL_2(5)$ and $K\cap R=Z(K)=Z$. Note that $Z$ is a normal subgroup of order $2$, hence is central in $R$. Consider $\bar G=G/Z$. Then $\bar G=\bar K \times \bar R$ and hence $\bar R$ is a group of Chillag Mann type, since $\bar K$ is simple and has irreducible real non-linear characters of both odd and even degree. This means that $\bar R=\bar O \times \bar H$ for $O \in Hall_{2'}(R)$ and $H \in Syl_2(R)$, note that $\bar H$ is of Chillag-Mann type. We have that $R$ is $2$-closed, hence $R=H \rtimes O$. Clearly, $O$ acts trivially on $H/Z$, so $H=C_H(O)Z\le C_H(O) Z(R)\cap H$. It follows that $O$ centralizes $H$ and $R=H\times O$. By Dedekind modular law $HK\cap O\le HK\cap R\le H(K\cap R)\le H$ and hence $HK\cap O\le H\cap O=1$. This means that $G=(KH)\times O$ and $K\cap H=Z$ that has order $2$. If $H$ and $K$ commute, then $KH=K\Ydown H$. Suppose by contradiction that $[H,K]=Z$, hence there is $aZ \in H/Z$ that acts non-trivially by conjugation on $K/Z$. But this is impossible since $KH/Z=\bar K\times \bar H$. We now prove that $H$ is of Chillag-Mann type. Suppose that this is not the case, so there is $\phi \in Irr_{rv}(H)$ such that $\phi(1)>1$. Since $\bar H$ is of Chillag-Mann type, we have that $ Z \not \le \ker{\phi}$, so $\phi_Z=\phi(1)\lambda$, with $\lambda \neq 1_Z$. On the other hand, if $\theta$ is the unique character of $K$ of degree $6$, then $Z \not \le \ker{\theta}$ and $\theta_Z=\theta(1)\lambda$. Now, $KH=K\times H/N$ where $N=\{(z,z)\mid z \in Z\}$ (see \cite[I9.10]{huppertI}) and $\psi=\theta\times \phi \in Irr_{rv}(K\times H)$. Moreover $\psi_N=\phi(1)\theta(1)\lambda^2=\phi(1)\theta(1)1_N$, so it follows that $N \le \ker{\psi}$ and $\psi \in Irr_{rv}(KH)$. If $\chi=\psi\times1_O$, then $\chi \in Irr(G)$ takes real values and has composite degree, impossible. Since $SL_2(5)$ does not satisfy the hypotheses, we have that $Z<H$. Point $ii)$ follows.
\end{proof}
We remark that in Problem $4.4$ of \cite{ctfg}, we can find a stronger version of the argument used in the Proof above.

	As a consequence, we get Theorem \ref{theoremB}.
	\begin{cor} \label{nonsolvrealdegs}
		Let $G$ a non-solvable group and suppose that $cd_{rv}(G)$ consists of prime-power numbers. Then either $cd_{rv}(G)=cd_{rv}(L_2(8))$ or $cd_{rv,2'}(G)=cd_{rv,2'}(A_5)$.
	\end{cor}
	\begin{proof} Apply Theorem \ref{noedgesnonsolvmain}. In case $i)$ there is nothing to prove. Suppose $ii)$, we have that $G=(KH)\times O$ with $O$ of odd order, $K=G^{(\infty)}$ and $H$ is a normal $2$-subgroup. Call $S$ the simple section $KH/H$, hence $S\simeq A_5$. Take $\chi \in Irr_{rv}(G)$ a real non-linear character of odd degree. Hence $\chi(1)=p^n$ with $p$ odd and $\chi$ is a character of $HK$ since, by Lemma \ref{dps1.4}, $O \le \ker (\chi)$. The degree of every irreducible constituent of $\chi_H$ divides $(\abs{H},\chi(1))=1$, hence $\chi_H=e\sum_i \lambda_i$ for $\lambda_i \in Lin(H)$. By hypothesis we have that $
	\chi(1)=p^*>1$ for an odd prime $p$ and by \cite[11.29]{ctfg} we have that $\chi(1)/\lambda(1)$ divides $[HK:H]=\abs{S}$, where $S\simeq A_5$. Hence $p\le \chi(1) \le \abs{S}_p$, the $p$-part of the number $\abs{S}$, that is equal to $p$ if $p$ is an odd prime. It follows that $\chi(1)=p$. We have proved that $cd_{rv,2'}(G) \subseteq \{3,5\}=cd_{rv,2'}(A_5)$. The right-to-left inclusion follows observing that $A_5$ is a quotient of $G$.% If $(S,p) \neq (PSL_2(8),3)$, then $\abs{S}_p=p$ and hence $\chi(1)=p$. Suppose $(S,p) = (PSL_2(8),3)$, in this case $\abs{S}_3=9$. The simple group $S$ has a real character of degree $9$, hence, if $\chi(1)=3^2$, we are done. Suppose that $\chi(1)=3$, we have that $\chi_H=e(\lambda_1+\lambda_2+\lambda_3)$ and $G$ permutes the $\lambda_i$. The group $H$ is contained in the kernel of the action, hence we have a non-trivial morphism from $KH/H \to S_3$. But $KH/H$ a non-abelian simple group and this is impossible.
	\end{proof}
	
\end{document}